\documentclass[11pt,reqno]{amsart}

\usepackage{amssymb, amsmath, amsthm, amsfonts, mathrsfs}
\usepackage[dvipsnames]{xcolor}
\usepackage{comment}
\usepackage[normalem]{ulem}

\usepackage{hyperref}

\newtheorem{definition}{Definition}[section]
\newtheorem{theorem}[definition]{Theorem}

\newtheorem{proposition}[definition]{Proposition}
\newtheorem{lemma}[definition]{Lemma}
\newtheorem{remark}[definition]{Remark}

\newtheorem{assumption}{Assumption}

\newcommand{\R}{\mathbb{R}}

\newcommand{\N}{\mathbb{N}}

\newcommand{\di}{\mathrm{d}}

\newenvironment{nouppercase}{%
  \renewcommand{\uppercasenonmath}[1]{}}{}
\title{{M}c{K}ean {SDE}s with singular coefficients}
\author{Elena Issoglio and Francesco Russo}

\date{26th June 2022}

\address[Elena Issoglio]{Dipartimento di Matematica `G.\ Peano', Universit\'a di Torino}
\email[Corresponding author]{elena.issoglio@unito.it}

\address[Francesco Russo]{Unit\'e de Math\'{e}matiques appliqu\'{e}es, ENSTA Paris, Institut Polytechnique de Paris}
\email{francesco.russo@ensta-paris.fr}

\begin{document}

\begin{abstract}
The paper investigates existence and uniqueness for a
stochastic differential equation (SDE) depending on the law density  of the solution, involving a Schwartz distribution. Those equations, known
 as McKean SDEs, are interpreted in the sense of a suitable singular martingale
problem. A key tool used in the investigation  is the study of the corresponding Fokker-Planck equation. 
\end{abstract}

\begin{nouppercase}
\maketitle
\end{nouppercase}

{\bf Key words and phrases.} Stochastic differential equations;
distributional drift; McKean; Martingale problem.

{\bf 2020 MSC}. 60H10; 60H30; 35C99; 35D99; 35K10.

\section{Introduction}
In this paper we are concerned with the study of  {\em singular McKean SDEs} of the form 
\begin{equation}\label{eq:McKean}
\left\{
\begin{array}{l}
X_t = X_0 + \int_0^t F(v(s, X_s))b(s, X_s)  \di s + W_t
\vspace{5pt}\\
v(t, \cdot) \text{ is the law density of }X_t,
\end{array}
\right.
\end{equation}
for some given initial condition $X_0$ with density  $v_0$. The terminology \emph{Mc\-Kean} refers to the fact that the coefficient of the SDE depends on the law of the solution process itself, while {\em singular} reflects the fact that one of the coefficients is a 
 Schwartz distribution. 
 The main aim of this paper is to solve the singular McKean problem \eqref{eq:McKean}, that is, to define rigorously the meaning of equation \eqref{eq:McKean} and to find a (unique) solution to the equation. The key novelty is the  {\em Schwartz distributional nature} of the drift, which is encoded in the term $b$. 

 The problem is $d$-dimensional, in particular the process $X$
 takes values in $\R^d$, the function $F$ is $F:\R\to \R^{d\times n}$, the  term $b$ is formally
 $b:[0,T]\times\R^d \to \R^n$ and $W$ is a $d$-dimensional Brownian motion, where $n,d$ are two integers. 
We assume that $b(t, \cdot)\in \mathcal C^{(-\beta)+}(\R^n)$ for some $0<\beta<1/2$ (see below for the definition of Besov spaces $\mathcal C^{-\beta}(\R^n)$), which means that $b(t, \cdot)$ is a Schwartz distribution and thus  the term $b(t, X_t)$, as well as its product with $F$, are only formal at this stage. The function $F$ is nonlinear.

\vspace{5pt}

The term $(s,x,v) \mapsto F(v(s,x)) b(s,x)$ in equation \eqref{eq:McKean} is a special case of a
general drift $(s,x,v) \mapsto f(s,x,v)$. When $f$ is a  function, equation \eqref{eq:McKean}
was studied by several authors.
For example \cite{JourMeleard} studies existence and uniqueness of the solution under several regularity assumptions on the drift, while \cite{lieber-oudjane-russo}  requires $f$ to be  Lipschitz-continuous with respect to the variable $v$, uniformly in time and space, and measurable with respect to time and space. We also mention \cite{BarbuRockSIAM}, where the authors obtain existence of the solution when assuming  that the drift is a measurable function. For other past contributions see  \cite{papierUgolini}.

Different settings of McKean-Vlasov problems   have been considered by other authors where the pointwise dependence on the density is replaced by a smoother dependence on the law, typically of Wasserstein type, and the Lipschitz property for the coefficients has been relaxed. From this perspective, the equations are not singular in our sense.
For example in \cite{deRaynal} the author considers McKean-Vlasov equations with coefficients $b$ and $\sigma$ which depend on the law of the process in a relatively smooth way, but are  H\"older-continuous in time and space. Later on in  \cite{HuangWangSPA} the authors considered SDEs where both the drift and the diffusion coefficient are of McKean type, with a Wasserstein dependence on the law, and where the drift satisfies a Krylov-R\"ockner $L^p$-$L^q$-type dependence. 
Independently   \cite{RocknerZhang} considered in particular SDEs with a McKean drift of the type $t\mapsto \int_{\mathbb R^d}b(X_t, y) \mu_{X_t}(dy)$ where $\mu_{X_t}$ is the law of $X_t$, and $b$ is some measurable function and $\sigma =\sqrt2$. 
 In \cite{HuangWang2021}, the authors study McKean-Vlasov SDEs with drift discontinuous under Wasserstein distance.

In the literature we also find  some contributions on \eqref{eq:McKean} with $F\equiv 1$, i.e.\ when there is no dependence on the law $v$ but the drift $b$ is a Schwartz distribution. In this case  equation \eqref{eq:McKean} becomes an SDE with singular drift. Ordinary SDEs with distributional drift were investigated
by several authors, starting from \cite{frw1, frw2, bass_chen, russo_trutnau07}
in the one-dimensional case.
In the multi-dimensional case it was studied by \cite{flandoli_et.al14} with $b$ being a Schwartz distribution living in a fractional Sobolev space of negative order (up to $-\tfrac12)$. 
Afterwards, \cite{cannizzaro} extended the study to a smaller negative order (up to $-\tfrac23$) and formulated the problem as a martingale problem. We also mention \cite{issoglio_russoMP}, where the singular SDE is studied
as a martingale problem, with the same setting as in the present paper (in particular the drift belongs to a negative Besov space rather than a fractional Sobolev space). 
Backwards SDEs with  similar singular coefficients have also been studied, see \cite{issoglio_jing16, issoglio_russo20}.

The main analytical tool in the works cited above is  an associated singular PDE  (either Kolmogorov or Fokker-Planck). In the McKean case,
the relevant PDE associated to equation \eqref{eq:McKean} is the nonlinear Fokker-Planck equation
\begin{equation}\label{eq:FPpde}
\left\{
\begin{array}{l}
\partial_t v = \frac12 \Delta v -\text{div}(\tilde F( v)b)\\
v(0) = v_0,
\end{array}
\right.
\end{equation}
where
\begin{equation} \label{Ftilde}
  \tilde F(v) := v F(v).
\end{equation}
  PDEs with similar (ir)regular coefficients were stu\-died in the past, see for example \cite{flandoli_et.al14, issoglio13} for the study of singular Kolmogorov equations. 
One can then use results on existence, uniqueness and continuity of the solution to the PDE (e.g.\ with respect to the initial condition and the coefficients) to infer results about the stochastic equation. For example in \cite{flandoli_et.al14}, the authors use the singular Kolmogorov PDE to define the meaning of the solution to the SDE and find a unique solution.

Let us remark that the PDEs mentioned above are a classical tool in the study of McKean equations when the dependence on the law density of the process is pointwise, which is the case in the present paper where we have $F(v(t,x))$. There is, however, a large body of literature that studies Mc\-Kean equations where the drift depends on the law more regularly, typically it is assumed to be Lipschitz-continuous with respect to the Wasserstein metric.
In this case the McKean equation is treated with different techniques than the ones explained above, in particular it is treated with probabilistic tools.
This is nowadays a well-known approach, for more details see for example the recent books by Carmona and Delarue \cite{carmona-delarueI, carmona-delarueII},
see also \cite{LOR1, LOR2}.

\vspace{5pt}

Our contribution to the literature is twofold. The first and main novel result concerns the notion of solution to the singular McKean equation \eqref{eq:McKean} (introduced in Definition \ref{def:solMcK}) and 
 its existence and uniqueness (proved in  Theorem \ref{thm:McKsol}).  
The second contribution is the study of the singular Fokker-Plank equation \eqref{eq:FPpde}, in particular we find a unique solution $v\in C([0,T]; C^{\beta+})$ in the sense of Schwartz distributions, see Theorem \ref{thm:esistunic} for  existence and uniqueness.

\vspace{5pt}
The paper is organised as follows. In Section \ref{sc:prelim} we introduce the notation and recall some useful results on semigroups and Besov spaces. We also recall briefly some results on the singular martingale problem. In  Section \ref{sc:FPpde} we study the singular Fokker-Planck PDE \eqref{eq:FPpde}.  Then we consider a mollified version of the PDE and the SDE in Sections  \ref{sc:regularisedPDE} and \ref{sc:regularisedSDE}, respectively. Finally in Section \ref{sc:McKean} we use the mollified PDEs and SDEs and their limits to study  \eqref{eq:McKean} and we prove our main theorem of existence and uniqueness of  a solution to \eqref{eq:McKean}. 
In Appendix \ref{app:frGronwall} we recall a useful fractional Gronwall's inequality. 
In Appendix \ref{app:inductive} we show a characterization of  continuity  and compactness in inductive spaces.

\section{Setting and useful results}\label{sc:prelim}
\subsection{Notation and definitions}
Let us use the notation $C^{0,1}:=C^{0,1}([0,T]\times \R^d)$ to indicate the space of jointly continuous functions with gradient in $x$ uniformly continuous in $(t,x)$.  By a slight abuse of notation we use the same notation
$C^{0,1}$  for functions which are $\R^d$-valued. 
When $f:\R^d \to \R^d$ is differentiable, we denote by $\nabla f$ the matrix given by $(\nabla f)_{i,j} = \partial_i f_j$. When $f: \R^d \to \R$ we denote the Hessian matrix of $f$ by Hess$(f)$.

We denote by $\mathcal S=\mathcal S(\mathbb R^d )$  the space of Schwartz functions on $\mathbb R^d$ and by $\mathcal S'=\mathcal S'(\mathbb R^d )$ the space of Schwartz distributions. 
 For $\gamma\in\mathbb R$ we denote by  $\mathcal C^\gamma = \mathcal C^\gamma(\mathbb R^d)$  the Besov space or H\"older-Zygmund space and by $\|\cdot\|_\gamma$ its norm, more precisely
\begin{equation*}
\mathcal C^\gamma := \left\{ f\in \mathcal S' : \| f\|_\gamma := \sup_{j\in\mathbb N} 2^{j\gamma} \|\mathcal F^{-1}(\varphi_j \mathcal F f)\|_\infty \right\},
\end{equation*} 
 where $\varphi_j$ is a partition of unity and $\mathcal F$ denotes the Fourier transform. For more details see for example \cite[Section 2.7]{bahouri}.
 We recall that for $\gamma'<\gamma$ one has $\mathcal C^\gamma\subset\mathcal C^{\gamma'}$. 
If $\gamma \in \R^+ \setminus \mathbb N$ then the space coincides with the classical H\"older space of functions which are $ \left \lfloor{\gamma}\right \rfloor $-times differentiable and such that the $ \left \lfloor{\gamma}\right \rfloor $th derivative is $ (\gamma - \left \lfloor{\gamma}\right \rfloor  )$-H\"older continuous. For example if $\gamma\in(0,1)$  the classical $\gamma$-H\"older norm 
\begin{equation}\label{eq:holder}
 \|f\|_{\infty} + \sup_{x\neq y, |x-y|<1} \frac{|f(x)-f(y)|}{|x-y|^\gamma},
\end{equation}
is an equivalent norm in $\mathcal C^\gamma$.  With an abuse of notation we use $\|f\|_{\gamma} $ to denote \eqref{eq:holder}. 
For this and for more details see, for example, \cite[Chapter 1]{triebel10} or \cite[Section 2.7]{bahouri}. 
Notice that we use the same notation $\mathcal C^\gamma$ to indicate  $\R$-valued functions but also $\R^d$ or $\R^{d\times d}$-valued  functions.
It will be clear from the context which space is needed.

We denote by $C_T \mathcal C^\gamma$ the space of continuous functions on $[0,T]$ taking values in $\mathcal C^\gamma$, that is $C_T \mathcal C^\gamma:= C([0,T]; \mathcal C^\gamma)$. 
For any given $\gamma\in \R$ we denote by $\mathcal C^{\gamma+}$ and $\mathcal C^{\gamma-}$  the spaces given by
\[
\mathcal C^{\gamma+}:= \cup_{\alpha >\gamma} \mathcal C^{\alpha} ,  \qquad  
\mathcal C^{\gamma-}:= \cap_{\alpha <\gamma} \mathcal C^{\alpha}.
\]
Notice that $\mathcal C^{\gamma+}$ is an inductive space. 
We will also use the spaces $C_T C^{\gamma+}:=C([0,T]; \mathcal C^{\gamma+})$, recalling that  $f\in C_T C^{\gamma+} $ if and only if there exists $\alpha>\gamma $ such that $f\in C_T C^{\alpha}$, see Lemma \ref{lm:inductive} in  Appendix \ref{app:inductive} for a proof of the latter fact.

Similarly, we use the metric space   $C_T C^{\gamma-}:=C([0,T]; \mathcal C^{\gamma-})$, meaning that  $f\in C_T C^{\gamma-} $ if and only if  for any $\alpha<\gamma $ we have $f\in C_T C^{\alpha}$.
Notice that if $f$ is continuous and such that $\nabla f \in C_T \mathcal C^{0+}$ then $f\in C^{0,1}$.

Let $(P_t)_t$ denote the semigroup generated by $\tfrac12\Delta$ on $\mathcal S$, in particular for all $\phi\in\mathcal S$ we define $(P_t \phi) (x):= \int_{\mathbb R^d} p_t(x-y) \phi(y)\mathrm dy$, where the kernel $p$ is the usual heat kernel 
\begin{equation}\label{eq:pt}
p_t(z) =  \frac{1}{(2\pi t)^{d/2}} \exp\{ -\frac{|z|^2}{t}\}.
\end{equation}    It is easy to see that $P_t: \mathcal S \to \mathcal S$. 
Moreover we can extend it to $\mathcal S'$  by dual pairing (and we denote it with the same notation for simplicity). 
One has $\langle P_t \psi, \phi\rangle = \langle \psi,  P_t \phi\rangle $ for each $\phi\in\mathcal S$ and  $\psi\in\mathcal S'$, using the fact that the kernel is symmetric.

\begin{lemma}\label{lm:heat}
Let $g:[0,T] \to \mathcal S'(\R^d)$ be continuous and $w_0\in \mathcal S'(\R^d)$. The unique (weak) solution of 
\[
\left\{
\begin{array}{l}
\partial_t w = \frac12 \Delta w +g\\
w(0) = w_0
\end{array}
\right.
\] 
is given by  
\begin{equation}\label{eq:heat}
P_t w_0 + \int_0^t P_{t-s} g(s) \di s, \quad t\in[0,T].
\end{equation}
By weak solution we mean, for every $\varphi \in \mathcal S(\R^d)$ and $t\in[0,T]$ we have $ \langle w(t), \varphi\rangle = \langle w_0, \varphi\rangle + \int_0^t \langle w(s),  \frac12 \Delta \varphi\rangle  \di s +  \int_0^t \langle  g(s),  \varphi\rangle  \di s$.
\end{lemma} 
 \begin{proof}
 The fact that \eqref{eq:heat} is a solution is done by inspection. The uniqueness is a consequence of Fourier transform. 
  \end{proof}

We denote by $\Gamma$ the usual Gamma function defined as  $\Gamma(\theta) = \int_0^\infty t^{\theta-1} e^{-t} \di t$ for $\theta>0$.

In the whole article  the letter $c$ or $C$ will denote a generic constant which may change from line to line.
 
\subsection{Some useful results} In the sections below, we are interested in the action of $P_t$ on elements of Besov spaces $\mathcal C^\gamma$. These estimates are known as \emph{Schauder's estimates} (for a proof we refer to \cite[Lemma 2.5]{catellier_chouk}, see also \cite{gubinelli_imkeller_perkowski}
for similar results).
 \begin{lemma}[Schauder's estimates]\label{lm:schauder}
Let $f\in \mathcal C^\gamma $ for some $\gamma \in \mathbb R$. Then for any $\theta\geq 0$ there exists a constant $c$ such that
\begin{equation}\label{eq:Pt}
\|P_t f\|_{\gamma + 2 \theta} \leq c t^{-\theta} \|f\|_\gamma,
\end{equation}  
for all $t>0$.

Moreover let $\theta\in(0,1)$. 
For $f\in \mathcal C^{\gamma + 2\theta }$  we have 
\begin{equation}\label{eq:Pt-I}
\|P_t f-f\|_{\gamma} \leq c t^{\theta} \|f\|_{\gamma+2\theta }.
\end{equation}
\end{lemma}
Notice that from \eqref{eq:Pt-I} it readily follows that if $f\in \mathcal C^{\gamma + 2 \theta}$ for some $0<\theta<1$, then for $t>s>0$ we have 
\begin{equation}\label{eq:PcontC}
\|P_tf-P_s f\|_{\gamma}\leq c (t-s)^{\theta} \|f\|_{\gamma+ 2\theta }.
\end{equation}
In other words, this means that if $f\in  \mathcal C^{\gamma+2\theta}$ then $P_\cdot f\in C_T \mathcal C^\gamma$ (and in fact it is $\theta$-H\"older continuous in time). 
We also recall that Bernstein's inequalities hold (see \cite[Lemma 2.1]{bahouri} and \cite[Appendix A.1]{gubinelli_imkeller_perkowski}), that is for $\gamma \in \mathbb R$ there exists a constant $c>0$ such that
\begin{equation}\label{eq:nabla}
\|\nabla g\|_{\gamma} \leq c  \|g\|_{\gamma+1},
\end{equation}  
for all $g\in \mathcal C^{1+\gamma}$. 
Using Schauder's and Bernstein's inequalities we can easily obtain a useful estimate on the gradient of the semigroup, as we see below. 
\begin{lemma}\label{lm:nablaP}
Let $\gamma\in \mathbb R$ and $\theta \in (0,1)$. If $g\in \mathcal C^\gamma$ then for all $t>0$ we have $\nabla (P_tg) \in \mathcal C^{\gamma +2\theta -1}$ and
\begin{equation}\label{eq:nablaP}
\|\nabla (P_tg)\|_{\gamma+2\theta-1} \leq c t^{-\theta}  \|g\|_{\gamma}.
\end{equation}  
\end{lemma}

The following is an important estimate which allows to define the so called
{\it pointwise product} between certain distributions and functions, which is based on Bony's estimates. For details see \cite{bony} or \cite[Section 2.1]{gubinelli_imkeller_perkowski}. Let   $f \in \mathcal C^\alpha$ and $g\in\mathcal C^{-\beta}$ with $\alpha-\beta>0$ and $\alpha,\beta>0$. Then the
{pointwise product} $ f \, g$ is well-defined as an element of $\mathcal C^{-\beta}$ and  there exists a constant $c>0$ such that 
\begin{equation}\label{eq:bony}
\| f \, g\|_{-\beta} \leq c \| f \|_\alpha \|g\|_{-\beta}.
\end{equation} 
Moreover if $f$ and $g$ are continuous functions defined on $[0,T]$ with values in the above Besov spaces,
one can easily show that the product is also continuous with values in  $\mathcal C^{-\beta}$, and 
\begin{equation}\label{eq:bonyt}
\| f \, g\|_{C_T\mathcal C^{-\beta}} \leq c \| f \|_{C_T \mathcal C^\alpha} \|g\|_{C_T \mathcal C^{-\beta}} .
\end{equation}

\subsection{Assumptions} 
We now collect the assumptions on the distributional term $b$, the nonlinearity  $ F$ and $\tilde F$
(see \eqref{Ftilde})
and on the initial condition $v_0$ that will be used later on in order for  PDE \eqref{eq:FPpde} to be well-defined and for the  McKean-Vlasov problem \eqref{eq:McKean} to be solved. 
\begin{assumption} 
 \label{ass:param-b} 
Let $0<\beta<1/2$ and  $b\in C_T \mathcal C^{(-\beta)+}$. In particular $b\in C_T \mathcal C^{-\beta}$.
\end{assumption}

In the following result we construct a sequence $b^n$ using the heat semigroup and prove certain properties.
\begin{proposition} \label{pr:L24}
  Let $b$ as in Assumption \ref{ass:param-b}.
 Let us define a sequence $(b^n)$ such that, for any fixed $t\in[0,T]$ and for all $n\geq1$ we have
$$b^n(t, \cdot):= \phi_n \ast b(t, \cdot),$$ 
where $ \phi_n (x) := p_{1/n}(x)$ and $p$ is the Gaussian kernel defined in \eqref{eq:pt}.

\begin{itemize}
\item [(i)] For each $n$,  $b^n$ is globally bounded, together with all its space derivatives.
\item [(ii)]
For each $n$, $t\mapsto b^n(t,\cdot)$ is continuous in $\mathcal C^{\gamma}$ for all $\gamma>0$. In particular $b^n \in C_T \mathcal C^{(-\beta)+}$.
\item [(iii)] We have the convergence   $b^n\to b$ in $C_T\mathcal C^{-\beta}$. 
\end{itemize}
\end{proposition}
\begin{proof}
If $\psi \in \mathcal S'$ then $\phi_n\ast \psi = P_{1/{n}} \psi$, thus  we have $b^n(t,\cdot) =  P_{1/{n}} b (t,\cdot) $. 
\begin{itemize}
\item [(i)] We have 
\[
\|P_{1/{n}} b(t)\|_{ \gamma} \leq c  \left(\frac1 {n}\right)^{-\frac{\gamma+\beta}2} \|b(t)\|_{-\beta}, 
\] 
for any $\gamma>0$ by Lemma \ref{lm:schauder}.
\item [(ii)]  For any $t,s\in[0,T]$  we have 
\begin{align*}
\|b^n(t, \cdot)-b^n(s, \cdot)\|_{\gamma}
=&
\|P_{1/n} b(t, \cdot)-P_{1/n} b(s, \cdot)\|_{\gamma} \\
= &
\|P_{1/n} (b(t, \cdot)- b(s, \cdot))\|_{\gamma}\\
\leq & c \left(\frac1{n}\right)^{-\frac{\gamma+\beta}2}
 \| b(t, \cdot)- b(s, \cdot)\|_{-\beta},
\end{align*}
having used estimate \eqref{eq:Pt} in Lemma \ref{lm:schauder} (with $\theta =\frac{\gamma+\beta}2$). The  conclusion now follows. 
\item[(iii)]  For $t\in[0,T]$ we have, using  \eqref{eq:Pt-I} in Lemma \ref{lm:schauder}
\begin{align*}
\|b^n(t, \cdot)-b(t, \cdot)\|_{-\beta}
=&
\|P_{1/n} b(t, \cdot)-b(t, \cdot)\|_{-\beta} \\
\leq & c \left(\frac1{n}\right)^{\frac{\beta-\beta'}2}
 \| b(t, \cdot)\|_{-\beta'},
\end{align*}
 for some $\beta'<\beta$ such that $b\in \mathcal C^{-\beta'} $, which exists by Assumption \ref{ass:param-b}. Now we  take the sup over $t\in[0,T]$ and we have $\|b^n-b\|_{C_T\mathcal C^{-\beta}} \to 0$ as $n \to \infty$, since $\beta-\beta'>0$.
\end{itemize}
\end{proof}

\begin{assumption}\label{ass:F}
Let $F$ be Lipschitz and bounded. 
\end{assumption} 

\begin{assumption}\label{ass:Ftilde}
 Let $\tilde F(z) := z F(z)$ be globally Lipschitz.
\end{assumption} 

We believe that Assumption \ref{ass:Ftilde} is unnecessary. Indeed by Assumption~\ref{ass:F} one gets that $\tilde F$ is locally Lipschitz with linear growth. This condition could be sufficient to show that a solution PDE \eqref{eq:FPpde}  exists, for example using techniques similar to the ones appearing in \cite[Proposition 3.1]{issoglio19} and \cite[Theorem 22]{lieber-oudjane-russo}. However we assume here $\tilde F $ to be Lipschitz to improve the readability of the paper. 

\begin{assumption}\label{ass:v0}
Let $v_0\in \mathcal C^{\beta+}$. 
\end{assumption}

\begin{assumption}\label{ass:udensity}
Let  $v_0$ be a bounded probability density.  
\end{assumption}

\subsection{The singular Martingale Problem}\label{ssc:MP}
We conclude this section with a short recap of useful results from \cite{issoglio_russoMP}, where the authors consider the Martingale Problem for SDEs of the form 
\begin{equation}\label{eq:SDE}
X_t = X_0 + \int_0^t B(s, X_s)  \di s + W_t, \quad  X_0\sim \mu,
\end{equation}
where $B$ satisfies Assumption \ref{ass:param-b} (with $b=B$)  and $\mu$ is a given probability measure. Notice that this  SDE can be considered as the linear counterpart of the McKean-Vlasov problem \eqref{eq:McKean}, which can be obtained for example by  `fixing' a suitable function $v$ and considering $B = F(v)b$ in the SDE in  \eqref{eq:McKean}.

First of all, let us recall the definition of the operator $\mathcal L$ associated to SDE \eqref{eq:SDE} given in \cite{issoglio_russoMP}. The operator $\mathcal L$ is defined as  
\begin{equation}\label{eq:L}
\begin{array}{lcll}
\mathcal L  :  &\mathcal D_{\mathcal L}^0 &\to &\{\mathcal S'\text{-valued integrable functions}\}\\
& f & \mapsto & \mathcal L f:=   \dot f +  \frac12 \Delta f + \nabla f \, B,
\end{array}
\end{equation}
where 
 $$\mathcal D_{\mathcal L}^0 : = C_T D \mathcal C^{\beta+} \cap C^1([0,T]; \mathcal S'), 
 $$
and  $ D \mathcal C^{\gamma} = \{h: \R^d \to \R  \text{ differentiable such that} \nabla h \in \mathcal  C^\gamma \}$.
Here $f: [0,T]\times  \R^d \to \R$ and the function $\dot f:[0,T]\to \mathcal S'$ is the time-derivative.   Note also that $\nabla f \, B$ is well-defined using  \eqref{eq:bony} and Assumption \ref{ass:param-b}. 
The Laplacian $\Delta$ is intended in the sense of distributions.  Notice that the identity functions $\text{id}_i (x) = x_i$ for any $i=1, \ldots, d$ belong to $\mathcal D_{\mathcal L}^0 $ and we have ${\mathcal L} (\text{id}_i) = b_i$.

Next we give the definition of solution to the martingale problem in \cite[Definition 4.3]{issoglio_russoMP}: a couple  $(X, \mathbb P)$  is a {\em solution to the martingale problem with distributional drift $B$ and initial condition $\mu$} (for shortness, solution of MP  with drift $B$ and i.c.\ $\mu$) if and only if for every $f \in \mathcal D_{\mathcal L}$
\begin{equation}\label{eq:MP}
f(t, X_t) - f(0, X_0) - \int_0^t (\mathcal L f) (s, X_s) \di s
\end{equation}
is a local martingale under $\mathbb P$. The domain  $\mathcal D_{\mathcal L} $ is given by
\begin{equation}\label{eq:D}
\begin{array}{ll}
\mathcal D_{\mathcal L} : = &  
\{ f \in \mathcal C_T \mathcal C^{(1+\beta)+}:  \exists g \in   C_T\bar{\mathcal C}_c^{0+}  \text{ such that }  \\ 
&   f \text{ is a weak solution of }  \mathcal  L f =g  \text{ and } f(T) \in \bar{\mathcal C}_c^{(1+\beta)+}    \},
\end{array} 
\end{equation} 
where  $\mathcal L$ has been defined in  \eqref{eq:L},  and the spaces $\bar {\mathcal C}_c^{\gamma+}$ are defined as  $\bar {\mathcal C}_c^{\gamma+} =\cup_{\alpha>\gamma} \bar{\mathcal C}_c^\alpha $ where  $\bar{\mathcal C}_c^\alpha$ is the closure of compactly supported functions of   ${\mathcal C}^{\alpha}$ with respect to the norm of $\mathcal C^{\alpha}$.
Finally  we recall that 
 $f\in C_T \bar{\mathcal C}_c^{\gamma+} $ if and only if there exists $\alpha>\gamma $ such that $f\in C_T \bar{\mathcal C}_c^{\alpha}$, 
by Remark  \ref{rm:inductive} part (ii).
 We say that {\em  the martingale problem with drift $B$   and i.c.\ $\mu$ admits uniqueness} if, whenever we have two solutions $(X^1, \mathbb P^1)$ and  $(X^2, \mathbb P^2)$ with  $X^i_0\sim \mu$, $i=1,2$, then the law
of $X^1$ under $ \mathbb P^1 $ equals the law of $X^2$ under $\mathbb P^2$. 
With this definition at hand, we show in \cite[Theorem 4.11]
{issoglio_russoMP} that  MP admits existence and uniqueness.

\section{Fokker-Planck singular PDE} \label{sc:FPpde}

This section is devoted to the study of the singular Fokker-Planck equation \eqref{eq:FPpde}, recalled here for ease of reading
\[
\left\{
\begin{array}{l}
\partial_t v = \frac12 \Delta v -\text{div}(\tilde F( v)b)\\
v(0) = v_0.
\end{array}
\right.
\] 
After introducing the notions of solution for this PDE (weak and mild, which turns out to be equivalent,
see Proposition \ref{pr:weak=mild}), we will show that there exists a unique  solution in Theorem \ref{thm:esistunic} with Banach's fixed point theorem.

Below we will need mapping properties of the  function $\tilde F$ when viewed as operator  acting on $\mathcal C^\alpha$, for some $\alpha\in(0,1)$. To this aim, we make a slight abuse of notation and denote by $\tilde F$ the function when viewed as an operator, that is for $f\in\mathcal C^\alpha$ we have $\tilde F(f):= \tilde F(f(\cdot))$. We sometimes omit the brackets and write $\tilde F f $ in place of $\tilde F (f)$. The result below on $\tilde F$ is taken from \cite{issoglio19}, Proposition 3.1 and equation (32).
\begin{lemma}[Issoglio \cite{issoglio19}]\label{lm:F}
Under Assumption \ref{ass:Ftilde} and if $\alpha\in(0,1)$ then
\begin{itemize}
\item   $\tilde F : \mathcal C^\alpha \to \mathcal C^\alpha$ and for all $f,g \in \mathcal C^\alpha$
\[ 
\|\tilde F f -\tilde F g \|_\alpha \leq c (1+ \| f\|_\alpha^2 +  \| g\|_\alpha^2  )^{1/2} \|f-g\|_\alpha;
\]
\item for all $f \in \mathcal C^\alpha$, $\| \tilde F f \|_\alpha \leq c ( 1+ \| f \|_\alpha)$.
\end{itemize}
\end{lemma}

This mapping property allows us to define  weak and mild solutions for the singular Fokker-Planck equation. 

\begin{definition}\label{def:wmsol} 
Let Assumptions \ref{ass:param-b}, \ref{ass:Ftilde} and \ref{ass:v0} hold and let $ v\in C_T\mathcal C^{\beta+}$. 
\begin{itemize}
\item[(i)] We say that  $v$ is a  \emph{mild solution} for the singular Fokker-Planck equation \eqref{eq:FPpde} if  the integral equation
\begin{equation}\label{eq:mildFP}
v(t)  = P_t v_0 - \int_0^t  P_{t-s} [ \text{{\em div}} (\tilde F(v(s)) b(s) ) ]  \di s, \  t\in[0,T]
\end{equation}
 is satisfied.
\item[(ii)] We say that  $v$ is a  {\em weak solution}  for the singular Fokker-Planck equation \eqref{eq:FPpde} if for all $\varphi\in \mathcal S (\R^d)$
 and all $t\in[0,T]$ we have
\begin{align}\label{eq:weakFP2}
\langle \varphi, v(t)\rangle  =  & \langle   \varphi, v_0\rangle + \int_0^t \langle \frac12 \Delta \varphi, v(s) \rangle \di s  
+ \int_0^t \langle  \nabla \varphi,  \tilde F( v )(s) b(s) \rangle \di s . 
\end{align}
\end{itemize}
Note that the term $\tilde F( v )(s) b(s) $ appearing in both items is well-defined as an element of $ \mathcal C^{-\beta}$ thanks to \eqref{eq:bony} and Assumption \ref{ass:param-b} together with Lemma \ref{lm:F}.
\end{definition}

\begin{proposition}\label{pr:weak=mild}
Let $ v\in C_T\mathcal C^{\beta+}$. The function $v$ is a weak solution of PDE \eqref{eq:FPpde} if and only if it is a mild solution.
\end{proposition}
\begin{proof}
This is a consequence of Lemma \ref{lm:heat} with $g(s):=- \text{div} (\tilde F(v(s)) b(s) )  $.
\end{proof}

Let us denote by $I$ the solution map for the mild solution of PDE \eqref{eq:FPpde}, that is for $v\in C_T \mathcal C^\alpha$ for some $\alpha\in(0,1)$ we have 
$$
I_t(v):= P_t v_0 - \int_0^t P_{t-s} [\text{div}(\tilde F (v(s)) b(s)) ] \di s.
$$ 
Then a mild solution of \eqref{eq:FPpde} is a solution of $v = I(v)$, in other words it is a fixed point of $I$.

We present now an a priori bound for mild solutions, if they exist. 
\begin{proposition}\label{pr:priori}
Let Assumptions \ref{ass:param-b}, \ref{ass:Ftilde} and \ref{ass:v0} hold.  Let $\alpha \in(\beta, 1-\beta)$.
If $v\in C_T\mathcal C^{\alpha}$ is such that $v =  I(v)$, then we have  
$$\|v\|_{C_T\mathcal C^{\alpha}}\leq K,
$$
where $K $ is a constant depending on  $ 
\|v_0\|_{\alpha},  \|b\|_{C_T\mathcal C^{-\beta}}, T$. Moreover $K$ is an increasing function of $\|b\|_{C_T\mathcal C^{-\beta}} $. 
\end{proposition}

\begin{proof}
Let
\begin{equation}\label{eq:H}
H_s(v):= \tilde F (v(s)) b(s)
\end{equation}
 for brevity.
 Using Bernstein's inequality \eqref{eq:nabla} we get
\[
\|\text{div}H_s(v) \|_{-\beta-1}\leq \sum_{i=1}^d \|   \frac{\partial}{\partial x_i} H_s(v) \|_{-\beta-1} \leq c  \sum_{i=1}^d \|   H_s(v)  \|_{-\beta} = cd \|   H_s(v)  \|_{-\beta}.
\]
Then using the definition of $H$ from \eqref{eq:H},  pointwise product property \eqref{eq:bony} 
(since $\alpha-\beta>0$)  and Lemma \ref{lm:F} we have
\begin{align}\label{eq:divH} 
\|\text{div}H_s(v) \|_{-\beta-1}  \leq c  \|    \tilde F (v(s)) \|_{\alpha} \| b(s)  \|_{-\beta}  \leq c (1+ \|  v(s) \|_{\alpha}) \| b(s)  \|_{-\beta},
\end{align}
where we recall that $c$ is now a constant that changes from line to line. 
Now using this, together with Schauder's estimates (Lemma \ref{lm:schauder} with $\theta:= \tfrac{\alpha+\beta+1}{2}$) and the fact that $\theta<1$,  for fixed $t\in[0,T]$, one obtains
\begin{align*}
\|v(t)\|_\alpha 
&\leq \| P_tv_0\|_\alpha + \int_0^t \|   P_{t-s} [\text{div} H_s(v)]  \|_\alpha \di s\\
&\leq c\| v_0\|_\alpha + \int_0^t c(t-s)^{-\frac{\alpha+\beta+1}{2}}\| \text{div}H_s(v)   \|_{-\beta-1} \di s\\
&\leq c\| v_0\|_\alpha + \int_0^t c(t-s)^{-\frac{\alpha+\beta+1}{2}} (1+ \|  v(s) \|_{\alpha}) \| b(s)  \|_{-\beta}  \di s\\
&\leq c\| v_0\|_\alpha + c  \| b  \|_{C_T\mathcal C^{-\beta}} \int_0^t (t-s)^{-\frac{\alpha+\beta+1}{2}} (1+ \|  v(s) \|_{\alpha})   \di s\\
&\leq c\| v_0\|_\alpha + c  \| b  \|_{C_T\mathcal C^{-\beta}}  T^{ \frac{1-\alpha-\beta}{2}} + c \| b  \|_{C_T\mathcal C^{-\beta}} \int_0^t (t-s)^{-\frac{\alpha+\beta+1}{2}} \|  v(s) \|_{\alpha}   \di s.
\end{align*}
Now by a generalised Gronwall's inequality (see Lemma \ref{lm:fr-gronwall}) we have 

\begin{align*}
\|  v(t) \|_{\alpha} &\leq  [ c\| v_0\|_\alpha + c  \| b  \|_{C_T\mathcal C^{-\beta}}  T^{ \frac{1-\alpha-\beta}{2}} ] E_\eta ( c  \| b  \|_{C_T\mathcal C^{-\beta}} \Gamma (\eta)t^\eta),
\end{align*}
with $\eta ={-\tfrac{\alpha+\beta+1}{2}+1} =\frac{1-\alpha-\beta}{2} >0$ and where $E_\eta$ is the Mittag-Leffler function,
  see Lemma \ref{lm:fr-gronwall}.
  Now taking the sup over $t\in[0,T]$ and using the fact that $E_\eta$ is increasing 
  we get 
\begin{align*}
 \|v&\|_{C_T\mathcal C^\alpha} \\
 &\leq \left[ c\| v_0\|_\alpha + c  \| b  \|_{C_T\mathcal C^{-\beta}}  T^{ \frac{1-\alpha-\beta}{2}} \right]  E_\eta \left( c  \| b  \|_{C_T\mathcal C^{-\beta}} \Gamma \left(\frac{1-\alpha-\beta}{2} \right)T^{\frac{1-\alpha-\beta}{2}}\right)\\
 &\leq \left[ c\| v_0\|_\alpha + c  \| b  \|_{C_T\mathcal C^{-\beta}}  T \right]  E_\eta \left( c  \| b  \|_{C_T\mathcal C^{-\beta}} \Gamma(1)T\right)\\
& =: K.
\end{align*}
This concludes the proof.
\end{proof}

We are interested in finding a mild solution of \eqref{eq:FPpde} according to Definition \ref{def:wmsol}, in the space $C_T\mathcal C^{\beta+}$.
Let us denote by $w(t) := v(t) -P_t v_0$  and by 
\begin{equation}\label{eq:J}
J_t(w)
:= \int_0^t P_{t-s} [\text{div}(\tilde F (w(s)+P_s v_0) b(s)) ] \di s .
\end{equation} 
Then the mild formulation \eqref{eq:mildFP} is equivalent  to 
\begin{equation}\label{eq:mildw}
w(t)=J_t(w),
\end{equation}
since $P_t v_0 \in C_T \mathcal C^{\beta+}$.

Then a mild solution of \eqref{eq:FPpde} is $v(t) = w(t) +P_tv_0 $ where $w$ is a solution of \eqref{eq:mildw}, in other words $w$ is a fixed point of the map $J$. 
For any  $\alpha \in \R$ we introduce a family of equivalent norms   in $ C_T\mathcal C^\alpha$ given by
 \begin{equation*}
 \|w \|^{(\rho)}_{C_T\mathcal C^\alpha} := \sup_{t\in[0,T]}  e^{-\rho t} \|w(t) \|_{\alpha} .
\end{equation*}
Consider then the $\rho$-ball in  $C_T \mathcal C^\alpha$ of radius $M$, given by 
\begin{equation}\label{eq:E}
E^\alpha_{\rho, M} := \{ v\in  C_T \mathcal C^\alpha : \|v \|^{(\rho)}_{C_T\mathcal C^\alpha} \leq M \}.
\end{equation}
Notice that these sets are closed with respect to the topology of $C_T\mathcal C^\alpha$, hence they are F-spaces, 
see \cite[Chapter 2.1]{dunford-schwartz},
with respect to the metric topology of $C_T\mathcal C^\alpha$.   
The $\rho$-equivalent norm generates  the $\rho$-equivalent metric with respect to the metric of $C_T\mathcal C^\alpha$, given by 
\begin{equation}\label{eq:drho}
\di_{\rho} (w,z):=   \|w(t) - z(t)\|^{(\rho)}_{C_T\mathcal C^\alpha} , \qquad \forall \rho\geq0,
\end{equation}
for any $w,z \in  C_T\mathcal C^\alpha$. 
Let $\rho_0>0$ and $ M_0>0$  be chosen arbitrarily. The $(E^\alpha_{\rho_0,M_0},\di_{\rho} )$ is again an F-space.

In the proofs below we will also use the notation 
\begin{equation}\label{eq:G}
G_s(w):= \tilde F (w(s)+P_sv_0) b(s)
\end{equation}
 for brevity.
In order to show that $J$ is a contraction, we first show that it maps balls into balls.
\begin{proposition}\label{pr:Jball}
Let Assumptions \ref{ass:param-b} and \ref{ass:Ftilde} hold.  Let $v_0\in \mathcal C^\alpha$ for some $\alpha \in(\beta, 1-\beta)$. 
Then there exists $\rho_0$ (depending on $\|b\|, \alpha$ and $ \beta$) and $M_*$ (depending on $\|b\|, \alpha$, $ \beta$,  $\rho_0$ and $\|v_0\|_\alpha$) such that 
\begin{equation}
J: E^\alpha_{\rho_0, M_0} \to E^\alpha_{\rho_0, M_0},
\end{equation}
for any $M_0\geq M_*$, where $E^\alpha_{\rho_0, M_0}$ have been defined in \eqref{eq:E}.
\end{proposition}
\begin{proof}
Let $w\in E^\alpha_{\rho_0, M_0}\subset  C_T\mathcal C^\alpha$, for some $\rho_0, M_0$ to be specified later. 
Using the definition of $J$, Schauder's estimate for the semigroup (Lemma \ref{lm:schauder}) and the definition of $G$ from \eqref{eq:G} we have 
\begin{align} \label{eq:Jrho}
e^{-\rho_0 t}\|J_t(w)\|_{\alpha} &\leq \int_0^t e^{-\rho_0 t}\|P_{t-s} [ \text{div} G_s(w)]\|_{\alpha} \di s\\ \nonumber
 &\leq  c \int_0^t e^{-\rho_0 t} (t-s)^{-\frac{\alpha+\beta+1}{2}} \| \text{div} G_s(w)\|_{-\beta-1} \di s. \nonumber
\end{align}
Now we use Bernstein's inequality \eqref{eq:nabla} to bound
\[
\|\text{div}G_s(w) \|_{-\beta-1}\leq \sum_{i=1}^d \|   \frac{\partial}{\partial x_i} G_s(w) \|_{-\beta-1} \leq c  \sum_{i=1}^d \|   G_s(w)  \|_{-\beta},
\]
and  using again the definition of $G$ from \eqref{eq:G}, the  pointwise product property \eqref{eq:bony} 
(since $\alpha-\beta>0$)  and Lemma \ref{lm:F} we have
\begin{align}\label{eq:divG}  \nonumber
\|\text{div}G_s(w) \|_{-\beta-1}  &\leq c  \|    \tilde F (w(s) +P_sv_0) \|_{\alpha} \| b(s)  \|_{-\beta}  \\  \nonumber
&\leq c (1+ \|  w(s)+P_sv_0 \|_{\alpha}) \| b(s)  \|_{-\beta}\\
&\leq c (1+ \|  w(s)\|_{\alpha}+\|P_sv_0 \|_{\alpha}) \| b(s)  \|_{-\beta}.
\end{align}
Now plugging  \eqref{eq:divG} into \eqref{eq:Jrho} we get
\begin{align}  \label{eq:Jv0}
&e^{-\rho_0 t}\|J_t(w)\|_{\alpha} \\
 & \leq c \int_0^t e^{-\rho_0 (t-s)} (t-s)^{-\frac{\alpha+\beta+1}{2}} e^{-\rho_0 s} (1+ \|  w(s) \|_{\alpha}+ \|  P_sv_0 \|_{\alpha})  \| b(s)  \|_{-\beta}  \di s . 
 \nonumber
\end{align}
Using the assumption that $w\in E^\alpha_{\rho_0, M_0} $ and choosing $M_0\geq \|v_0\|_\alpha$  we have that $\sup_{s\in[0,T]}e^{-\rho_0 s} (1+ \|  w(s) \|_{\alpha}+ \|P_s v_0\|_{\alpha}) \leq (1+2M_0) $ (since $P_s$ is a contraction)
thus  \eqref{eq:Jv0} gives 
\begin{align}  \label{eq:Jv01}
e^{-\rho_0 t}\|J_t(w)\|_{\alpha}  
 \nonumber
 & \leq c \| b  \|_{C_T\mathcal C^{-\beta}} (1+2M_0) \int_0^t e^{-\rho_0 (t-s)} (t-s)^{-\frac{\alpha+\beta+1}{2}}  \di s \\
 & \leq c \| b  \|_{C_T\mathcal C^{-\beta}} (1+2M_0) \Gamma(\theta) \rho_0^{-\theta},
\end{align}
where  
$$\theta := \frac{1-\alpha-\beta}{2} = - \frac{\alpha+\beta+1}{2}+1 $$
 is positive by Assumption \ref{ass:param-b} and by $\alpha\in(\beta, 1-\beta)$.
We want to choose $\rho_0$ and $M_0$ such that $\sup_{t \in [0,T]}
e^{-\rho_0 t}\|J_t(w)\|_{\alpha} \leq M_0$, for which it is enough that
\begin{align}
c \| b  \|_{C_T\mathcal C^{-\beta}} (1+2M_0) \Gamma(\theta) \rho_0^{-\theta} &\leq M_0\\
\iff \nonumber\\
c \| b  \|_{C_T\mathcal C^{-\beta}} \Gamma(\theta) \rho_0^{-\theta} &\leq M_0 (1- 2 c \| b  \|_{C_T\mathcal C^{-\beta}} \Gamma(\theta) \rho_0^{-\theta})\\
\iff \nonumber\\
\frac{c \| b  \|_{C_T\mathcal C^{-\beta}} \Gamma(\theta) \rho_0^{-\theta} }{(1- 2 c \| b  \|_{C_T\mathcal C^{-\beta}} \Gamma(\theta) \rho_0^{-\theta})}&\leq M_0, 
\end{align}
provided that the denominator is positive.  To do so, we pick $\rho_0$ large enough so that 
\begin{equation}\label{eq:rho0}
1- 2c \| b  \|_{C_T\mathcal C^{-\beta}} \Gamma(\theta)\rho_0^{-\theta} > 0.
\end{equation}
Then we set
\[
  M_*:=  \frac{c \| b  \|_{C_T\mathcal C^{-\beta}} \Gamma(\theta) \rho_0^{-\theta}}{\left (1- 2 c \| b  \|_{C_T\mathcal C^{-\beta}}  \Gamma(\theta) \rho_0^{-\theta} \right)} \vee  \|v_0\|_\alpha,
\]
where  $\rho_0$ has been chosen in \eqref{eq:rho0}. 
Then for any $M_0\geq M_*$, and with this choice of $\rho_0$  we have indeed  that $\|J(w)\|^{(\rho_0)}_{C_T \mathcal C^\alpha} \leq M_0$ and therefore if $w\in E^\alpha_{\rho_0, M_0}$ then $J(w) \in E^\alpha_{\rho_0, M_0}$ as wanted. 
\end{proof}

We show below that it is possible to choose $\rho$ large enough such that 	$J$ is a contraction on $E^\alpha_{\rho_0,M_0}$ under $\di_\rho$, with $\rho_0, M_0$  chosen according to Proposition \ref{pr:Jball}.
\begin{lemma}\label{lm:Jcontr}
Let Assumptions \ref{ass:param-b} and \ref{ass:Ftilde} hold.  Let $v_0\in \mathcal C^\alpha$ for some $\alpha \in(\beta, 1-\beta)$. Let $J$ be defined in \eqref{eq:J}. 
 Let $\rho_0, M_0$ be chosen according to Proposition \ref{pr:Jball}.
Then there exists a constant $C$ (depending on $ T, \|b\|_{C_T \mathcal C^{-\beta}}, \alpha, \beta ,\rho_0 $ and $ M_0$) such that for all $w,z \in E^\alpha_{\rho_0, M_0}$ it holds
\[
\di_\rho(J(w),J(z))\leq C \rho^{-\theta} \di_\rho(w,z),
\]
where $\theta:= \tfrac{1-\alpha-\beta}{2} >0$.\\
In particular, for $\rho$ large enough, we have that $J$ is a contraction. 
\end{lemma}
\begin{proof}
Let $w,z \in E^\alpha_{\rho_0,M_0}$. Using the definition of $\di_\rho$, of the solution map $J$ and of $G$ as in \eqref{eq:G}   we have
\begin{align}\label{eq:Jdrho}\nonumber
\di_\rho &(J(w),J(z)) =  \sup_{t\in[0,T]} e^{-\rho t} \|J_t(w) - J_t(z)\|_\alpha\\ \nonumber 
& \leq   c \int_0^t e^{-\rho t} (t-s)^{-\frac{\alpha+\beta+1}{2}} \| \text{div} ( G_s(w)- G_s(z))\|_{-\beta-1} \di s\\
&=  c \int_0^t e^{-\rho (t-s)} (t-s)^{-\frac{\alpha+\beta+1}{2}} e^{-\rho s}\| \text{div} ( G_s(w)- G_s(z))\|_{-\beta-1} \di s. 
\end{align}
By Bernstein's inequality \eqref{eq:nabla}, pointwise product property \eqref{eq:bony}, the contraction property of $P_t$, local Lipschitz property of $\tilde F$ from Lemma \ref{lm:F} and definition of $\rho$-equivalent metric we get
\begin{align}\label{eq:divGG}\nonumber
e^{-\rho s}\|  \text{div} &( G_s(w)- G_s(z))\|_{-\beta-1} \\ \nonumber
& \leq c e^{-\rho s} \|\tilde F (w(s) - P_s v_0)-\tilde F (z(s) - P_s v_0)\|_{\alpha} \| b(s)  \|_{-\beta}\\\nonumber
&  \leq  c (1+  \|  w(s)\|_{\alpha}+ \|  z(s)\|_{\alpha}+2\|v_0 \|_{\alpha})  e^{-\rho s}\|w(s)-z(s)\|_\alpha \| b(s)  \|_{-\beta}\\
&  \leq c (1+  2e^{\rho_0 T}M_0 +2 M_0) \di_\rho(w, z) \| b(s)  \|_{-\beta},
\end{align}
having used in the last line the fact that $\|v_0\|_\alpha\leq M_0$ by choice of $M_0$ and that for any $w \in E^\alpha_{\rho_0, M_0}$ one has 
$$\sup_{s\in[0,T]}\|w(s)\|_\alpha = \sup_{s\in[0,T]} e^{\rho_0 s} e^{-\rho_0 s} \|w(s) \|_\alpha \leq e^{\rho_0 T} \|w\|^{(\rho_0)}_{C_T \mathcal C^\alpha}\leq e^{\rho_0 T} M_0 .$$
Plugging \eqref{eq:divGG} into \eqref{eq:Jdrho} and using the Gamma function we get
\[
\di_\rho(J(w),J(z)) \leq  c (1+  2e^{\rho_0 T}M_0 +2 M_0)\| b(s)  \|_{-\beta}    \di_\rho(w,z) \Gamma (\theta) \rho^{-\theta},
\]
hence setting  
\[
C : = c (1+  2e^{\rho_0 T}M_0 +2 M_0)\| b(s)  \|_{-\beta}     \Gamma (\theta)
\]
we conclude. 
\end{proof}

We can now state and prove  existence and uniqueness of a mild solution $v\in C_T \mathcal C^\alpha$ of \eqref{eq:FPpde} using the equivalent equation \eqref{eq:mildw}.

\begin{proposition}\label{pr:exunalpha}
Let  Assumptions \ref{ass:param-b} and \ref{ass:Ftilde} hold.  Let $v_0\in \mathcal C^\alpha$ for some $\alpha \in(\beta, 1-\beta)$.  Then there exists a unique $v\in C_T \mathcal C^\alpha$ such that \eqref{eq:mildFP} holds. 
\end{proposition}

\begin{proof}
We show existence and uniqueness of $w$ solution of \eqref{eq:mildw} because this is equivalent to existence and uniqueness of a mild solution $v\in C_T\mathcal C^\alpha$ to \eqref{eq:FPpde}.

To show existence,  let $\rho_0, M_0$ be chosen according to Proposition \ref{pr:Jball}. We observe that $J$ is a contraction  in $E^{\alpha}_{\rho_0, M_0} \subset C_T \mathcal C^\alpha$  by Lemma \ref{lm:Jcontr}, hence  by Banach's fixed point theorem there exists a unique solution to $w=J(w)$ in $E^{\alpha}_{\rho_0, M_0} \subset C_T\mathcal C^\alpha$.

To show uniqueness, let $w^1,w^2\in C_T \mathcal C^\alpha$ be any two  solutions of \eqref{eq:mildw}. Let $\rho_0$ be chosen according to Proposition \ref{pr:Jball}. Then we  set $M_i:= \|w^i\|_{C_T\mathcal C^\alpha}^{(\rho_0)}$ and  we choose $M_0\geq \max\{M_1, M_2, M_*\}$, with $M_*$ from Proposition \ref{pr:Jball}, so that $w^i\in E^\alpha_{\rho_0, M_0}$. Thus by the contraction property of $J$ we have uniqueness, hence $w^1=w^2$.
\end{proof}

\begin{theorem}\label{thm:esistunic}
Let Assumptions \ref{ass:param-b}, \ref{ass:Ftilde} and \ref{ass:v0} hold.  Then there exists a unique mild solution $v\in C_T \mathcal C^{\beta+}$ of \eqref{eq:FPpde}.
\end{theorem}
\begin{proof}
{\em Existence.} Since $v_0 \in \mathcal C^{\beta+}$ by assumption, there exists $\alpha\in(\beta, 1-\beta)$ such that $v_0 \in \mathcal C^{\alpha}$. With such $\alpha$ by Proposition \ref{pr:exunalpha} we know that there exists a (unique) mild solution in $ C_T \mathcal C^{\alpha}$.

{\em Uniqueness.} Given two solutions $v^1, v^2 \in  C_T \mathcal C^{\beta+}$ there exist $\alpha_1, \alpha_2$ such that  $v^i \in C_T \mathcal C^{\alpha^i}$ for $i=1,2$. Then choosing $\alpha= \min\{\alpha^1, \alpha^2\}$ we have that  $v^i \in C_T \mathcal C^{\alpha}$ for $i=1,2$ and by uniqueness in $ C_T \mathcal C^{\alpha}$ from Proposition \ref{pr:exunalpha} we have that $v^1=v^2$.
\end{proof}

\begin{remark}
Notice that if we suppose that $v_0\in\mathcal C^{(1-\beta)-}$  in place of Assumption \ref{ass:v0} one gets that a solution $v$ exists in $ C_T \mathcal C^{(1-\beta)-} $. 
\end{remark}

\section{The regularised PDE and its limit}\label{sc:regularisedPDE}

Let Assumptions \ref{ass:param-b}, \ref{ass:F}, \ref{ass:Ftilde} and \ref{ass:v0}   hold throughout  this section.

  We consider the sequence $b^n$ introduced in  Proposition \ref{pr:L24}.
 When the term $b$ is replaced  by $b^n$, with fixed $n$, then  we get a smoothed PDE, that is,  we get  the Fokker-Planck equation 
\begin{equation}\label{eq:FPpden}
\left\{
\begin{array}{l}
\partial_t v^n = \frac12 \Delta v^n -\text{div}(\tilde F( v^n)b^n)\\
v^n(0) = v_0,
\end{array} \right.
\end{equation}
where we recall that  $\tilde F(v^n) = v^n F(v^n)$. For ease of reading, we recall that the mild solution of \eqref{eq:FPpden} is given by an element $v^n\in C_T \mathcal C^{\beta+}$ such that 
\begin{equation}\label{eq:mildvn}
v^n(t)  = P_t v_0 - \int_0^t  P_{t-s} [\text{div} (\tilde F(v^n(s)) b^n(s) ) ]  \di s.
\end{equation}

\begin{remark}\label{rm:vn}
We observe that, since $b^n \in  C_T \mathcal C^{(-\beta)+}$, then all results from Section \ref{sc:FPpde} are still valid, in particular the bound from Proposition \ref{pr:priori} and the existence and uniqueness result from Theorem \ref{thm:esistunic} still apply to \eqref{eq:FPpden}.
\end{remark}

At this point we introduce the notation and some useful results on a very similar semilinear PDE studied in \cite{lieber-oudjane-russo}. 
We consider the  PDE 
\begin{equation}\label{eq:FPpdeR}
\left\{
\begin{array}{l}
\partial_t u(t,x) = \frac12 \Delta u(t,x) -\text{div}(u(t,x) \mathfrak b(t,x,u(t,x))  )\\
u(0,x) = v_0( x),
\end{array} \right.
\end{equation} 
where $v_0$ is a bounded Borel function. We set 
\begin{equation}\label{eq:b}
\mathfrak b(t,x,z) := F(z)b^n (t,x).
\end{equation} 
Thanks to Assumptions \ref{ass:F} and properties of $b^n$ stated in Proposition \ref{pr:L24} item (i) we have that  the  term $\mathfrak b(t, x, z)$ is  uniformly bounded.
Below we recall a mild-type solution, introduced in \cite{lieber-oudjane-russo}, which we call here semigroup solution. We will show that  any semigroup solution is also a mild solution in Proposition \ref{pr:mild=}.
\begin{definition}
We will call a {\em semigroup  solution} of the PDE \eqref{eq:FPpdeR} a function $u\in L^\infty([0,T]\times \R^d)$ that satisfies the integral equation
\begin{align}\label{eq:mildR}
\nonumber
u(t,x) = & \int_{\R^d} p_t(x- y)  v_0(y) \mathrm dy \\
& - \sum_{j=1}^d \int_0^t \int_{\R^d} u(s,y)\mathfrak b_j(s, y, u(s, y)) \partial_{j} p_{t-s}(x-y) \mathrm dy \,\mathrm ds ,
\end{align}
where  $p$ is the Gaussian heat kernel introduced  in \eqref{eq:pt}.
\end{definition}
Notice that this definition is inspired by \cite[Definition 6]{lieber-oudjane-russo}, but we mo\-di\-fied it here to include the condition $u\in L^\infty([0,T]\times \R^d)$, rather than $u\in L^1([0,T]\times \R^d)$ (the latter as in \cite{lieber-oudjane-russo}, where moreover the solution is called `mild solution'). Indeed integrability of $u$ is sufficient for the integrals in the {semigroup solution} to make sense, because $\mathfrak b$ is also bounded and the heat kernel and its derivative are integrable. 

The first result we have on \eqref{eq:FPpdeR} is about uniqueness of the {semigroup solution} in $L^\infty([0,T]\times \R^d)$. This result is not included in \cite{lieber-oudjane-russo}, but we were inspired by proofs therein, in particular by the proof of \cite[Lemma 20]{lieber-oudjane-russo}.

\begin{lemma}\label{lm:Linfty}
 There exists at most one {semigroup  solution} of \eqref{eq:FPpdeR}.
\end{lemma}
\begin{proof}
First of all we remark that since $p_t(y)$ is the heat kernel then we have two positive constants $c_p, C_p $ such that 
\begin{equation}\label{eq:pq}
|\partial_{y_j} p_t(y)| \leq \frac {C_p}{\sqrt t } q_t(y),
\end{equation}
for all $j=1, \ldots, d$, where $q_t(y) = \left(\frac{c_p}{ t \pi} \right)^{d/2} e^{-c_p \frac{|y|^2}{t}}$ is a Gaussian probability density. 

Let us consider two {semigroup solutions} $u_1, u_2$ of \eqref{eq:FPpdeR}.
We  denote by $\Pi(u)$ the {semigroup solution} map, which is the right-hand side of \eqref{eq:mildR}. Notice that   $v_0\in L^\infty(\R^d)$ by Assumption \ref{ass:v0},  and the function $z\mapsto z\mathfrak b(t,x,z)$ is Lipschitz, uniformly in $t,x$ because $\tilde F$ is assumed to be Lipschitz in Assumption \ref{ass:Ftilde}.  
Using  this, together with the bound \eqref{eq:pq}, for fixed $t\in(0,T]$, we get
\begin{align*}
&\|\Pi(u_1)(t,\cdot) - \Pi(u_2)(t,\cdot)\|_{\infty}\\
& = \Big \|\sum_{j=1}^d \int_0^t \int_{\R^d} \Big(u_2(s,y)\mathfrak b_j(s, y, u_1(s, y)) -u_1(s,y) \mathfrak b_j(s, y, u_2(s, y))\Big)\\
& \qquad \qquad \qquad \qquad \cdot  \partial_{j} p_{t-s}(x-y) \mathrm dy \,\mathrm ds \Big \|_\infty\\
& \leq C \int_0^t \int_{\R^d} |u_1(s,y)-u_2(s,y) |  \frac {1}{\sqrt{ t-s} } C_p q_{t-s}(x-y)  \mathrm dy \,\mathrm ds\\
& \leq C  \int_0^t  \|u_1(s,\cdot)-u_2(s,\cdot) \|_\infty  \frac {1}{\sqrt{ t-s} }\mathrm ds \cdot  \int_{\R^d}  q_{t-s}(x-y)  \mathrm dy \\
&\leq C \int_0^t  \|u_1(s,\cdot)-u_2(s,\cdot) \|_\infty  \frac {1}{\sqrt{ t-s} }\mathrm ds.
\end{align*}  
Now, by an application of a fractional Gronwall's inequality (see Lemma \ref{lm:fr-gronwall})  we conclude that $\|u_1(t,\cdot)-u_2(t,\cdot) \|_\infty \leq 0$ for all $t\in[0,T]$, so in particular we have 
$$\|u_1 - u_2\|_{L^\infty([0,T]\times \R^d)} =0,$$
 hence the { semigroup solution} is unique in  $L^\infty([0,T]\times \R^d)$.
\end{proof}

At this point we want to compare the concept of mild solution and that of {semigroup solution}. Recall that $\mathfrak b(t, x, z) =  F( z ) b^n(t,x) $ 
so  in fact PDE \eqref{eq:FPpdeR} is exactly \eqref{eq:FPpden}. First we state and prove a preparatory lemma, where  $\mathfrak f$ is vector-valued and will be taken to be $ u(t,x)\mathfrak b(t,x,u(t,x)) $ for fixed $t$ in the following result.

\begin{lemma}\label{lm:P=p}
  Let $\mathfrak f \in L^\infty (\mathbb R^d; \R^d)$,
  $ t \in (0,T]$. 
  Then
\begin{equation} \label{EP=p}
P_t (\text{div } \mathfrak f )= \sum_{j=1}^d  \int_{\R^d}  \mathfrak f_j (y) \partial_j p_t(\cdot - y  )\mathrm  dy,
\end{equation}
almost everywhere.
\end{lemma}

\begin{proof} 
  We will show that the left-hand side (LHS) and the right-hand side
  (RHS) are the
  same object in $\mathcal S'$.  
     Notice that the heat kernel $p_t(x)$ is the same kernel associated to the semigroup $P_t$, namely if $\phi\in\mathcal S$, then $P_t \phi \in \mathcal S$ with $P_t \phi(x) = \int_{\R^d} p_t(x-y) \phi(y) \di y $. 
We now take the Fourier transform $\mathcal F$ in $\mathcal S'$  of both sides. 
The LHS gives
\begin{align*}
\mathcal F( P_t (\text{div } \mathfrak f )) &= \mathcal F( p_t \ast (\text{div } \mathfrak f ))\\
&= \mathcal F( p_t) \mathcal F  (\text{div } \mathfrak f )\\
&= \sum_{j=1}^d \mathcal F( p_t)  i \xi_j \mathcal F  (  \mathfrak f_j ).
\end{align*}
The RHS of \eqref{EP=p}, on the other hand, gives
\begin{align*}
\mathcal F(\sum_{j=1}^d  \int_{\R^d}  \mathfrak f_j (y) \partial_{j} p_t(\cdot - y) \mathrm dy )& = \sum_{j=1}^d \mathcal F(  \mathfrak f_j \ast \partial_{j} p_t )
\\
& = \sum_{j=1}^d \mathcal F(  \mathfrak f_j) \mathcal F( \partial_{j} p_t )\\
& = \sum_{j=1}^d \mathcal F(  \mathfrak f_j) i\xi_j \mathcal F(  p_t ).
\end{align*}
Notice that one should be careful that the products appearing above are classical products of an element of $\mathcal S'$ (like $\mathcal F f_j$) and an element of $\mathcal S$ (like $\xi \mapsto i \xi_j \mathcal F (p_t)(\xi)$).
\end{proof}

We are now ready to prove that any mild solution  is  a  semigroup solution.

\begin{proposition}\label{pr:mild=}
Any  mild solution $v^n$  of \eqref{eq:FPpden}  
 is  a  semigroup solution. 
\end{proposition}
\begin{proof}
Recall that $F(z) b^n(t,x) = \mathfrak b(t,x,z)$ by \eqref{eq:b}. 
For $v^n$ to be a semigroup solution it must be an a.e.\ 
bounded function that satisfies \eqref{eq:mildR}.
First we notice that, since $v^n$ is a mild solution, there exists $\alpha>\beta $  such that $v^n \in C_T\mathcal C^\alpha \subset L^\infty([0,T]\times \R^d)$ so the second term on the RHS of expression \eqref{eq:mildR} is well-defined.
We recall that by Assumption \ref{ass:F}, $ F$ is bounded and by
Proposition \ref{pr:L24} (i) also $b^n$ is bounded
 hence  $\mathfrak b$ is also bounded. 
Moreover by Assumption \ref{ass:v0} the initial condition $v_0 \in \mathcal C^{\beta+}\subset L^\infty ([0,T]\times \R^d) $ so also the first term  on the RHS of expression \eqref{eq:mildR} is well-defined.

Now we show that the two terms on the RHS of \eqref{eq:mildvn} are equal to
 the terms on the RHS of \eqref{eq:mildR}. We start with the initial condition
 term, which can be written as
\begin{align*}
(P_t v_0)(x) = \int_{\R^d} p_t(x- y) v_0(y) \mathrm dy,  
\end{align*}
since $p_t$ is the kernel of the semigroup $P_t$. 
For the second term we use Lemma \ref{lm:P=p} with $\mathfrak f= u \mathfrak b$ to get 
\begin{align*}
P_t (\text{div} &[u(t) F( u(t)) b^n(t,\cdot) ]) =P_t (\text{div} [u(t)\mathfrak b(t, u(t)) ])\\
&=\sum_{j=1}^d  \int_0^t \int_{\R^d}  u(s, y)   \mathfrak  b_j(s, y, u(s, y)) \partial_{j} p_t(\cdot - y  )\mathrm  dy \,\mathrm  ds
\end{align*}
and so \eqref{eq:mildR} becomes \eqref{eq:mildvn}, i.e.~the mild solution $v^n$ is also a semigroup solution. 
\end{proof}

\begin{remark}\label{rm:vnunique}
Let  $n$ be fixed. By Theorem \ref{thm:esistunic}
there is a  unique mild solution $v^n$ of \eqref{eq:FPpden} in $C_T \mathcal C^{\beta+}$.
\end{remark}

The next result establishes, in particular, the uniqueness of  the solution $v$ in  $C_T \mathcal C^{\beta+}$ and a continuity result with respect to $b\in C_T\mathcal C^{-\beta}$.

\begin{proposition}\label{pr:unique}
  \begin{itemize}
  \item[(i)] Let $b^1, b^2$ satisfy Assumption \ref{ass:param-b}. Let  
  $v^1$ (resp.\ $v^2$) be a  mild solution of \eqref{eq:FPpde} with $b=b^1$ (resp.\ $b=b^2$).
  For any   $\alpha\in (\beta,1-\beta)$ such that $v^1,v^2\in C_T \mathcal C^\alpha$, there exists a function
    $\ell_\alpha:\R^+\times \R^+\to \R^+$, increasing in the second variable,
     such that 
\[
\|v^1(t) - v^2(t)\|_{\alpha}  \leq \ell_\alpha(\|v_0\|_\alpha, \|b^1\| \vee\|b^2\|  )  \|  b^1 - b^2 \|_{C_T\mathcal C^{-\beta}},
\]
 for all $t\in[0,T]$.
\item[(ii)]  Let  $(b^m)_m$ be a sequence in $C_T \mathcal C^{(-\beta)+}$. Let $v^m$ be a mild solution of   \eqref{eq:FPpde} with $b=b^m$ and $v$ be a mild solution of \eqref{eq:FPpde}.
If $b^m\to b$ in  $C_T \mathcal C^{-\beta}$ then $v^m \to v$ in $C_T \mathcal C^{\beta+}$.
\end{itemize} 
\end{proposition}

\begin{proof}
  {\em Item (i).} 
   Let $v^1$ (resp.\ $v^2$) be a  solution in $C_T\mathcal C^{\beta+}$ to \eqref{eq:FPpde} with $b=b^1$ (resp.\ $b=b^2$); so there exists  $\alpha\in(\beta, 1-\beta) $ such that $v^1, v^2 \in C_T\mathcal C^{\alpha}$.
We fix $t\in[0,T]$.
Using Schauder's estimates and Bernstein's inequalities,
for the difference below
we get the  bound
\begin{align}\label{eq:vnm}
\notag
\|v^1(t) &- v^2(t)\|_{\alpha} 
= \left \| \int_0^t P_{t-s}\left( \text{div} [\tilde F (v^1(s)) b^1(s)- \tilde F (v^2(s)) b^2(s)] \right)  \di s \right \|_\alpha\\ \notag
\leq &  c\int_0^t (t-s)^{-\frac{\alpha+\beta+1}2} \left\| \text{div} [\tilde F (v^1(s)) b^1(s)- \tilde F (v^2(s)) b^2(s)] \right\|_{-\beta-1}   \di s \\ 
\leq &  c\int_0^t (t-s)^{-\frac{\alpha+\beta+1}2} \left\| \tilde F (v^1(s)) b^1(s)- \tilde F (v^2(s)) b^2(s) \right\|_{-\beta}   \di s.
\end{align}
Now, in order to bound the term inside the integral we use  the mapping properties of $\tilde F$ from Lemma \ref{lm:F}, the property \eqref{eq:bony}
of the pointwise product, and the fact that $v^1$ and $v^2$ are mild solutions. We get
\begin{align*}
 &\left\| \tilde F (v^1(s)) b^1(s)- \tilde F (v^2(s)) b^2(s) \right\|_{-\beta}  \\
= & \left \| \tilde F (v^1(s)) b^1(s)- \tilde F (v^2(s)) b^1(s) + \tilde F (v^2(s)) b^1(s) - \tilde F (v^2(s)) b^2(s) \right\|_{-\beta}  \\
\leq & \left\| [\tilde F (v^1(s))- \tilde F (v^2(s)) ] b^1(s) \right \|_{-\beta}  +\left  \|\tilde F (v^2(s)) [ b^1(s) - b^2(s)] \right \|_{-\beta}  \\
  \leq &  c\left \| \tilde F (v^1(s))- \tilde F (v^2(s)) \right \|_\alpha \| b^1(s) \|_{-\beta}  + c\left \|\tilde F (v^2(s))\right \|_\alpha \|  b^1(s) - b^2(s) \|_{-\beta}  \\
\leq & c \left  (1+ \| v^1(s)\|^2_\alpha+ \| v^2(s)\|^2_\alpha \right )^{1/2} \left \| v^1(s)- v^2(s)  \right \|_\alpha \| b^1(s) \|_{-\beta} \\ 
& + c(1+\|v^2(s)\|_\alpha)  \|  b^1(s) - b^2(s) \|_{-\beta}  \\
\leq & c \left (1+ \| v^1\|^2_{C_T\mathcal C^\alpha}+ \| v^2\|^2_{C_T\mathcal C^\alpha} \right )^{1/2} \| v^1(s)- v^2(s)  \|_\alpha \| b^1 \|_{C_T\mathcal C^{-\beta}} \\
 & + c\left (1+\|v^2\|_{C_T\mathcal C^\alpha} \right )  \|  b^1 - b^2 \|_{C_T\mathcal C^{-\beta}}.
 \end{align*}
At this point we use the a priori bound $K_1$ for $v^1$ (resp.\ $K_2$ for $v^2$) found in Proposition \ref{pr:priori}, which depends on  $\|v_0\|_\alpha $ and $\|b^1\|_{C_T\mathcal C^{-\beta}} $ (resp.\ $\|b^2\|_{C_T\mathcal C^{-\beta}} $) and is increasing with respect to the latter.  Thus we get 
 \begin{align*}
 &\left\| \tilde F (v^1(s)) b^1(s)- \tilde F (v^2(s)) b^2(s) \right\|_{-\beta}     \\
&\leq c \left (1+ K_1^2+ K_2^2 \right )^{1/2} \| v^1(s)- v^2(s)  \|_\alpha \| b^1 \|_{C_T\mathcal C^{-\beta}} \\
 &  + c\left (1+K_2 \right )  \|  b^1 - b^2 \|_{C_T\mathcal C^{-\beta}}  \\
 &\leq  \tilde \ell_\alpha \left(\|v_0\|_\alpha,  \|b^1\|_{C_T\mathcal C^{-\beta}} \vee \|b^2\|_{C_T \mathcal C^{-\beta}} \right)  \| v^1(s)- v^2(s)  \|_\alpha \\
  &+   \tilde \ell_\alpha \left(\|v_0\|_\alpha, \|b^1\|_{C_T\mathcal C^{-\beta}} \vee \|b^2\|_{C_T \mathcal C^{-\beta}} \right) \|  b^1 - b^2 \|_{C_T\mathcal C^{-\beta}},
\end{align*}
where $\tilde \ell_\alpha (\cdot, \cdot) $ is a function  increasing in the second variable.
Putting this into \eqref{eq:vnm} we get
\begin{align*}
&\|v^1(t) - v^2(t)\|_{\alpha} \\
&\leq   c\, \tilde \ell_\alpha \left(\|v_0\|_\alpha, \|b^1\|_{C_T\mathcal C^{-\beta}} \vee \|b^2\|_{C_T \mathcal C^{-\beta}} \right)  \|  b^1 - b^2 \|_{C_T\mathcal C^{-\beta}} T^{\frac{1-\alpha-\beta}2} \\
&+   \tilde \ell_\alpha \left(\|v_0\|_\alpha, \|b^1\|_{C_T\mathcal C^{-\beta}} \vee \|b^2\|_{C_T \mathcal C^{-\beta}} \right) \int_0^t (t-s)^{-\frac{\alpha+\beta+1}2} \| v^1(s)- v^2(s)  \|_\alpha \di s  ,
\end{align*}
and by a generalised Gronwall's inequality (see Lemma \ref{lm:fr-gronwall}) we get
\begin{align*}
&\|v^1(t) - v^2(t)\|_{\alpha}  \\
&\leq  c\, \tilde \ell_\alpha \left(\|v_0\|_\alpha, \|b^1\|_{C_T\mathcal C^{-\beta}} \vee \|b^2\|_{C_T \mathcal C^{-\beta}} \right) 
 \|  b^1 - b^2 \|_{C_T\mathcal C^{-\beta}} T^{\frac{1-\alpha-\beta}2}\\
 &\times E_{\frac{1-\alpha-\beta}2} \left(  \tilde \ell_\alpha \left(\|v_0\|_\alpha, \|b^1\|_{C_T\mathcal C^{-\beta}} \vee \|b^2\|_{C_T \mathcal C^{-\beta}} \right) \Gamma\left(\tfrac{1-\alpha-\beta}2 \right) T^{\frac{1-\alpha-\beta}2} \right) \\
&= :   \ell_\alpha \left(\|v_0\|_\alpha, \|b^1\|_{C_T\mathcal C^{-\beta}} \vee \|b^2\|_{C_T \mathcal C^{-\beta}} \right)   \|  b^1 - b^2 \|_{C_T\mathcal C^{-\beta}},
\end{align*}
where $ \ell_\alpha (\cdot, \cdot) $ is again a function  increasing in the second variable.

{\em Item (ii).} Let $(b^m)_m$ be a sequence in $C_T\mathcal C^{(-\beta)+}$.  Let us assume that $v^m$ is the  unique solution of \eqref{eq:FPpde} with $b=b^m$ by Theorem \ref{thm:esistunic}. Moreover, by Proposition \ref{pr:exunalpha},  such $v^m$ lives in $C_T\mathcal C^\alpha$, where $\alpha$ depends only on $v_0$, hence not on $m$. Let $v$ be the unique solution of \eqref{eq:FPpde}. 
 We  apply  Item (i) with $b^1=b^m$ and $b^2 =b$ to get 
\begin{equation}\label{eq:bmb}
\|v^m(t) - v(t)\|_{\alpha} \leq  \ell_\alpha\left (\|v_0\|_\alpha, \|  b^m \|_{C_T\mathcal C^{-\beta}}\vee  \|  b \|_{C_T\mathcal C^{-\beta}} \right )  \|  b^m - b \|_{C_T\mathcal C^{-\beta}}.
\end{equation}
We have $\sup_m \|b^m\|_{C_T\mathcal C^{-\beta}}<\infty$ because $b^m\to b$ in $C_T\mathcal C^{-\beta}$,  
and 
\[
 \ell_\alpha\left (\|v_0\|_\alpha, \|  b^m \|_{C_T\mathcal C^{-\beta}} \vee  \|  b \|_{C_T\mathcal C^{-\beta}} ) \right) 
\leq \ell_\alpha \left( \|v_0\|_\alpha, \sup_m \|  b^m \|_{C_T\mathcal C^{-\beta}}\vee   \|  b \|_{C_T\mathcal C^{-\beta}} ) \right) 
\]
because $\ell_\alpha(\|v_0\|_\alpha, \cdot)$ is increasing. Therefore plugging this into \eqref{eq:bmb} we have 
\[
\|v^m(t) - v(t)\|_{\alpha} \leq  c \|b^m- b\|_{C_T\mathcal C^{-\beta}},
\]
where $c:=\ell_\alpha \left (\|v_0\|_\alpha, \sup_m \|  b^m \|_{C_T\mathcal C^{-\beta}} \vee \|  b \|_{C_T\mathcal C^{-\beta}}  \right) $. Thus taking the sup over $t$ we get that $v^m\to v$ in $C_T\mathcal C^\alpha$ if $b^m\to b$ in $C_T\mathcal C^{-\beta}$, which implies the convergence of $v^m \to v$ in  $C_T\mathcal C^{\beta+}$ because $\alpha>\beta$.
\end{proof}

\section{The regularised SDEs}\label{sc:regularisedSDE}

In this section we consider the regularised version of the McKean SDE introduced in \eqref{eq:McKean}, when $b$ is replaced by a $b^n$ defined in Proposition \ref{pr:L24}, for fixed $n$. 
We focus on the SDE 
\begin{equation}\label{eq:McKeann}
\left\{
\begin{array}{l}
X^n_t = X_0 + \int_0^t F(v^n(s, X^n_s))b^n(s, X^n_s)  \di s + W_t
\vspace{5pt}\\
v^n(t, \cdot) \text{ is the law density of }X^n_t,
\end{array}
\right.
\end{equation}
where $X_0$ is a given random variable  distributed according to $v_0$.  
In order to show existence and uniqueness of a solution of \eqref{eq:McKeann} and its link to the mild (and semigroup) solution $v^n$ of \eqref{eq:FPpden}, we make use of Theorems 12 and 13  from \cite{lieber-oudjane-russo}, as we see below.

\begin{proposition}\label{pr:MKR}
Let Assumptions  \ref{ass:F}, \ref{ass:Ftilde} and \ref{ass:udensity} hold. Let $( W_t)_{t\in[0,T]}$ be a Brownian motion on  some
given probability space. Let $b^n:[0,T]\times \R^d \to \R$ be a bounded Borel function  and let $ X_0 \sim  v_0$. 
\begin{itemize}
\item[(i)] There exists a couple $(X^n,v^n)$  with $v^n$ bounded, verifying \eqref{eq:McKeann}.
\item[(ii)] Given two solutions $(X^n,v^n)$ and $(\hat X^n,\hat v^n)$ of \eqref{eq:McKeann} with $v^n$ and $\hat v^n$ bounded, then   $(X^n,v^n)=(\hat X^n,\hat v^n)$.
\item[(iii)] If $(X^n,v^n)$ is a solution to \eqref{eq:McKeann} with $v^n$ bounded, then $v^n$ is a semigroup solution of \eqref{eq:FPpdeR}. 
\end{itemize}
\end{proposition}

\begin{proof}
 We observe that \eqref{eq:McKeann} is the special case of equation (1) in \cite{lieber-oudjane-russo} when $\Lambda =0, b_0 =0 , (a_{i,j}) = I, \Phi = I$ and $\mathfrak u_0$ has a density $v_0$ with respect to the Lebesgue measure. Notice that all assumptions in Theorems 12 and 13 are satisfied. Indeed, the drift  $\mathfrak b(t,x,z): = F (z)b^n(t,x)$ is bounded and Lipschitz with respect to $z$
  because $F$ is Lipschitz and bounded by Assumption \ref{ass:F} and $b^n$ is   bounded by  Proposition \ref{pr:L24} item (i).
  
{\em Item (i).}  We apply the result  \cite[Theorem 13 point 3]{lieber-oudjane-russo}. 
In fact, the authors  forgot to emphasize 
that the $v^n$ can be chosen to be bounded  (contrary to Theorem 13 point 1 where they emphasized it).

{\em Item (ii).}
 We apply the result  \cite[Theorem 13 point 2]{lieber-oudjane-russo}. 

{\em Item (iii).} We apply the result  \cite[Theorem 12 point 1]{lieber-oudjane-russo} to get that $v^n$ is a weak solution of \eqref{eq:FPpdeR}. Under \cite[Assumption C]{lieber-oudjane-russo}\footnote{which postulates uniqueness of weak solutions for $\partial_t u = L^* u, u_0=0$ in the class of measure valued functions, which is true if $L^* =\Delta$, see \cite[Remark 7]{lieber-oudjane-russo}.}, weak and  semigroup solutions are equivalent, see \cite[Proposition 16]{lieber-oudjane-russo}.
\end{proof}

\section{Solving the McKean problem} \label{sc:McKean}

 Let  Assumptions  \ref{ass:param-b},   \ref{ass:F}, \ref{ass:Ftilde}, \ref{ass:v0} and \ref{ass:udensity}  be standing assumptions in this section. 
For ease of reading, we recall  the problem at hand, which was illustrated in \eqref{eq:McKean}. We want to solve the  McKean equation 

\begin{equation}\label{eq:McKean2}
\left\{
\begin{array}{l}
X_t = X_0 + \int_0^t F(v(s, X_s))b(s, X_s)  \di s + W_t
\vspace{5pt}\\
v(t, \cdot) \text{ is the law density of }X_t,
\end{array}
\right.
\end{equation}
for some given initial condition $X_0 \sim v_0$. 
The corresponding  Fokker-Planck singular equation (already introduced in \eqref{eq:FPpde} and recalled here for ease of reading) is  
\begin{equation}\label{eq:FPpde2}
\left\{
\begin{array}{l}
\partial_t v = \frac12 \Delta v -\text{div}(\tilde F( v)b)\\
v(0) = v_0,
\end{array}
\right.
\end{equation}
where $\tilde F(v) := v F(v)$, to which we  gave a proper meaning and which we solved in Section \ref{sc:FPpde}.

\begin{remark}
In \cite{JourMeleard} the authors investigate the propagation of chaos for McKean SDE \eqref{eq:McKean2} with smooth coefficients and initial condition, using a system of moderately interacting particles. The corresponding  system in our singular framework appears to be
\begin{align*}
\di X_t^{i,N} =& F\bigg (\frac1N \sum_{j=1}^N \phi_\epsilon(X^{j,N}_t - X^{i,N}_t) \bigg) b(t, X_t^{i,N})  \di t
+ \di W^i_t, \quad i=1, \ldots, N,
\end{align*}
 where $\phi_\epsilon$ is a mollifier converging to $\delta_0$.

 We observe that the above equations can be considered as a $dN$-dimensional SDE 
 \[
\di X_t = B(t, X_t) \di t + \di W_t,
 \]
 with singular drift $B = (B_1, B_2, \ldots, B_N)^\top$ where 
\[
B_i (t, x^1, x^2, \ldots, x^N) = F\big(\frac1N \sum_{j=1}^N \phi_\epsilon(x^{j}_t - x^{i}_t) \big) b(t, x^{i})
\]
and each $x^j\in \mathbb \R^d$. This singular SDE is well-defined using \cite{issoglio_russoMP} (see also \cite{flandoli_et.al14}) because $B(t)\in C^{(-\beta)+}(\mathbb R^{dN})$ since $b(t)\in  C^{(-\beta)+} (\mathbb R^d)$ and $ F \circ \phi_\epsilon$ is Lipschitz and bounded (since both $F$ and $\phi_\epsilon$ are Lipschitz and bounded).

We leave the study of this system and its behaviour when $N\to \infty$ to future research.  
\end{remark}

\begin{definition} \label{def:solMcK}
A  {\em solution (in law) of the  McKean problem}  \eqref{eq:McKean2} is a triple  $(X,\mathbb P, v) $ such that $\mathbb P$ is a probability measure on some measurable space $(\Omega, \mathcal F)$, the function $v$ is defined on $[0,T] \times \R^d$ and belongs to $ C_T\mathcal C^{\beta+}$, the couple $(X, \mathbb P)$ is a solution to the martingale problem with distributional drift $B(t,\cdot) :=  F(v(s, \cdot))b(s, \cdot) $, and $v(t, \cdot)$ is the law density of $X_t$.  

 We say that {\em the McKean problem \eqref{eq:McKean2}  admits uniqueness} if, whenever we have two solutions $(X, \mathbb P, v)$ and $(\hat X, \hat{\mathbb P}, \hat v)$, then $v=\hat v$ in $C_T\mathcal C^{\beta+}$ and the law of $X$ under $ \mathbb P$ equals the law of $\hat X$ under $\hat {\mathbb P}$.
\end{definition}

Using the tools developed in the previous sections, in Theorem \ref{thm:McKsol} we will construct  a solution $(X, \mathbb P, v) $ to the McKean problem \eqref{eq:McKean2} and show that this solution is  unique.
We first recall two useful results from \cite{issoglio_russoMP}. Let us consider a distributional drift $B\in C_T\mathcal C^{(-\beta)+}$ that satisfies Assumption \ref{ass:param-b} with $b=B$.

 The first result concerns  convergence in law when the distributional drift $B$ is approximated by a sequence of smooth functions $B^n$.
  This result is crucial to show existence of the McKean equation.

\begin{proposition}(Issoglio-Russo, \cite[Theorem 4.16]{issoglio_russoMP}). \label{pr:tight}
Let $B$ satisfy Assumption \ref{ass:param-b}.  Let $(B^n)$ be a sequence in $C_T\mathcal C^{(-\beta)+}$ converging to $B $ in $C_T\mathcal C^{-\beta}$. Let $(X, \mathbb P)$ (respectively $(X^n, \mathbb P^n)$) be a solution  to the (linear) MP with distributional drift $B$ (respectively $B^n$).
Then the sequence $(X^n, \mathbb P^n)$  converges in law to  $(X, \mathbb P)$. In particular, if $B^n$ is a bounded function
(which also belongs to $C_T\mathcal C^{(-\beta)+}$)
  and $X^n$ is a (strong) solution of 
\begin{equation*}
X^n_t = X_0 + \int_0^t B^n(s, X^n_s) \di s + W_t,
\end{equation*}
then $X^n$ converges to $(X, \mathbb P)$ in law.
\end{proposition}

The second result is the fact that the law of the solution $X$ to a (linear) martingale problem with distributional drift $B$ solves the Fokker-Planck equation in the weak sense. This result is crucial to show uniqueness of the McKean equation.

\begin{proposition}(Issoglio-Russo, \cite[Theorem 4.14]{issoglio_russoMP}).\label{pr:FPlawMP}
Let $B$ satisfy Assumption \ref{ass:param-b}. 
Let $(X,\mathbb P)$ be a solution to the martingale problem with distributional drift $B$. Let $v(t, \cdot)$ be the law density of $X_t$ and let us assume that $v\in C_T\mathcal C^{\beta+}$. Then $v$ is  a weak solution (in the sense of Definition \ref{def:wmsol} part (ii)) of the Fokker-Plank equation 
\begin{equation*}
\left\{
\begin{array}{l}
\partial_t v = \frac12 \Delta v -\text{\em div}(v B )\\
v(0) = v_0.
\end{array}
\right.
\end{equation*}
\end{proposition}
We can now state and prove the main result of this paper. 

\begin{theorem}\label{thm:McKsol}
Let  Assumptions  \ref{ass:param-b},  \ref{ass:F}, \ref{ass:Ftilde}, \ref{ass:v0} and \ref{ass:udensity} hold. Then there exists a solution $(X, \mathbb P, v)$ to the McKean problem \eqref{eq:McKean2}. Furthermore, the McKean problem admits uniqueness according to Definition \ref{def:solMcK}. 
\end{theorem}

\begin{proof}
\emph{Existence.} Let us consider the sequence $(b^n)\to b $ defined in Proposition \ref{pr:L24}. 
The corresponding smoothed McKean problem is 
\begin{equation}\label{eq:MKn}
\left\{
\begin{array}{l}
X^n_t = X_0 + \int_0^t F(v^n(s, X^n_s))b^n(s, X^n_s)  \di s + W_t,
\vspace{5pt}\\
v^n(t, \cdot) \text{ is the law density of }X^n_t.
\end{array}
\right.
\end{equation}
By Proposition \ref{pr:MKR} part (i) we have a solution $(X^n, v^n)$ of \eqref{eq:MKn} where $v^n$ is bounded and $X^n$ is a (strong)  solution  of $\di X^n = B^n(t, X^n_t) \di t + \di W_t; X^n_0=X_0$ on some fixed probability space $(\Omega, \mathcal F, \mathbb P)$, with  $B^n: = F(v^n)b^n $. By Proposition  \ref{pr:MKR} part (iii) we have that $v^n$ is a semigroup solution of \eqref{eq:FPpden}. On the other hand, we know by  Remark \ref{rm:vn} that  a mild solution $u^n$ of the same equation exists. By Proposition \ref{pr:mild=} we know that $u^n$ is a semigroup solution and moreover it is bounded (because it is a mild solution).
By uniqueness of semigroup solutions (see Lemma \ref{lm:Linfty}) we have $v^n = u^n$.

Now we notice that $B^n = F(v^n)b^n $ converges to $B:= F(v)b $ in $C_T\mathcal C^{-\beta}$ because of  \eqref{eq:bonyt},
the linearity of the pointwise product, the Lipschitz property of $F $, Lemma \ref{lm:F}, the convergence $b^n\to b$ by Proposition \ref{pr:L24} item (iii) and  the convergence $v^n\to v $ by Proposition \ref{pr:unique}.  By Lemma \cite[Lemma 4.14]{issoglio_russoMP} we know that $(X^n, \mathbb P)$ is also a solution to the MP with distributional drift $B^n$ and initial condition $X_0$, hence applying Proposition \ref{pr:tight} we have that $X^n\to X$ in law (as $B^n\to B$), and since $v^n$ is the law density of $X^n$ we have that $v$ must be the law density of $X$.

\emph{Uniqueness.}
Suppose that we have two solutions of the McKean problem \eqref{eq:McKean2}, $(X^1, \mathbb P^1, v^1)$  and $(X^2, \mathbb P^2, v^2)$.
By definition we know that $(X^i, \mathbb P^i)$ is a solution to the (linear) martingale problem with distributional drift $B^i:= F(v^i)b$. Thus by Proposition \ref{pr:FPlawMP} we have that $v^i$ is a weak solution to the Fokker-Planck equation  
\[
\left\{
\begin{array}{l}
\partial_t v^i = \frac12 \Delta v^i -\text{div}(v^i F(v^i)b )\\
v^i(0) = v_0,
\end{array}
\right.
\]
which is exactly PDE \eqref{eq:FPpde2}. Item (ii) in Proposition \ref{pr:unique} guarantees uniqueness of the mild solution of \eqref{eq:FPpde2}  and Proposition \ref{pr:weak=mild} ensures that weak and mild solutions of the Fokker-Planck equation are equivalent, hence $v^1=v^2=:v$. Note that it is crucial the fact that $v^i\in C_T \mathcal C^{\beta+}$. 
This implies that   $(X^i, \mathbb P^i)$ are both solutions of the same (linear) martingale problem with distributional drift $B:= F(v) b$, so by uniqueness of the solution of MP (see Section \ref{ssc:MP}) we conclude that  the law of $X^1$ under $ \mathbb P^1 $ equals the law of $X^2$ under $\mathbb P^2$.
\end{proof}

\appendix

\section{A generalised Gronwall's inequality}\label{app:frGronwall}

Here we recall a useful generalised Gronwall's inequality (or fractional Gronwall's inequality). For a proof see \cite[Corollary 2]{ye}.

\begin{lemma}\label{lm:fr-gronwall}
Suppose $\eta > 0$, $a(t)$ is a nonnegative function locally integrable on $0\leq t < T$ (some $T \leq \infty$) and nondecreasing on $[0,T)$. Let $g(t)$ be a nonnegative, nondecreasing continuous function defined on $0 \leq t < T$, $g(t)\leq M$ (constant), and suppose $f(t)$ is nonnegative and locally integrable on $0\leq t < T$ with
\[
f(t) \leq a(t) +g(t) \int_0^t (t-s)^{\eta-1} f(s) \di s
\]
on this interval. 
Then
\[
f(t)\leq  a(t)E_{\eta}(g(t)\Gamma (\eta)t^\eta),
\]
where $E_\eta$ is the Mittag-Leffler function defined by $E_\eta(z) =\sum_{k=0}^\infty \frac{z^k}{\Gamma(k\eta+1)}$.
\end{lemma}

\begin{remark} \label{R_issoglio19}
  In \cite{issoglio19}, the end of the proof of Proposition 4.1 incorrectly uses Gronwall's lemma. The proper argument should instead cite a generalised Gronwall's inequality, like the one stated above. 
\end{remark}

\section{Compactness and continuity in inductive spaces}\label{app:inductive}
This Appendix is devoted to the proof of a continuity result in inductive spaces. We show in two steps that a function belongs to $C_T\mathcal C^{\gamma+}$ if and only if it belongs to $C_T \mathcal C^\alpha$ for some $\alpha>\gamma$. 

The first step is about compactness of sets in inductive spaces $\mathcal C^{\gamma+}$.

\begin{lemma}\label{lm:compactenss}
Let $\gamma>0$. A set  $K\subset \mathcal C^{\gamma+}$ is a compact in $\mathcal C^{\gamma+}$ if and only if there exists  $\alpha>\gamma$ such that $K\subset \mathcal C^\alpha$ and $K$ is a compact in $\mathcal C^\alpha$.
\end{lemma}
\begin{proof}
``$\Rightarrow$''. 
Let $K\subset \mathcal C^{\gamma+}$ be a compact. For any $x\in K$, we know that $x\in C^{\alpha(x)}$ for some $\alpha(x)>\gamma$ and we pick  an arbitrary open neighbourhood  $V(x)$ in $\mathcal C^{\alpha(x)}$. Thus $V(x)$ is an open set of $\mathcal C^{\gamma+}$.  We have
$K \subset \cup_{x\in K} V(x)$, and since $K$ is compact in $\mathcal C^{\gamma+}$ there exists a finite subcovering $K\subset \cup_{i=1}^N V(x_i)$. Let $\alpha:= \min_{i=1, \ldots, N} \alpha (x_i)$. 
Thus $K\subset \mathcal C^\alpha$. 
Next we show that $K$ is also a compact in  $\mathcal C^\alpha$ for the chosen $\alpha$. Let $(O_\nu)_\nu$ be any open covering of $K$ in $\mathcal C^\alpha$, that is $K\subset \cup_\nu O_\nu$. Each $O_\nu$ is an open set of $\mathcal C^\alpha$ thus also of $\mathcal C^{\gamma+}$, therefore $(O_\nu)_\nu$ is also an open covering of $\mathcal C^{\gamma+}$, thus there exists a finite covering. 

``$\Leftarrow$''. 	
Let $K$ be a compact in $\mathcal C^\alpha$, for some $\alpha>\gamma$. The inclusion $K\subset \mathcal C^{\gamma+}$ is obvious. Now let us take an open covering of $K$ in  $ \mathcal C^{\gamma+}$, that is $K\subset \cup_\nu O_\nu$, where each $O_\nu$ is an open set in  $\mathcal C^{\gamma+}$. Since $K\subset \mathcal C^\alpha$, then  $K\subset \cup_\nu (O_\nu \cap \mathcal C^{\alpha})$. Finally we notice that since $O_\nu$  is an open set in $\mathcal C^{\gamma+}$, by trace topology we have that $O_\nu \cap \mathcal C^{\alpha}$ is an open set of $\mathcal C^\alpha$ (because $\mathcal C^\alpha$ is a closed set of $\mathcal C^{\gamma+}$). Thus we can extract a finite subcovering in $\mathcal C^\alpha$, which will be also a finite subcovering of $K$ in $\mathcal C^{\gamma+}$.
\end{proof}

Next we show  the continuity result. 
\begin{lemma}\label{lm:inductive}
Let $\gamma>0$. Then $C_T \mathcal C^{\gamma+} =  \cup_{\alpha>\gamma}C_T \mathcal C^{\alpha}$.
\end{lemma}

\begin{proof}
The inclusion  $\supseteq$ is obvious.\\ 
Next we show the inclusion  $\subseteq$. Let $f:[0,T] \to \mathcal C^{\gamma+}$ be continuous. We have to find $\alpha>\gamma$ such that $f\in C_T \mathcal C^\alpha$.
 Let $E_f:= \{ f(t), t\in [0,T] \}$, which is a compact in $\mathcal C^{\gamma+}= \cup_{\alpha>\gamma} \mathcal C^{\alpha}$ since it is the image of the compact $[0,T]$ via $f$ which is continuous. By Lemma  \ref{lm:compactenss} there exists  $\alpha>\gamma$ such that $E_f$ is a compact in $\mathcal C^\alpha$, in particular, $f:[0,T]\to \mathcal C^\alpha$. It remains to show that $f(t_n)\to f(t_0)$ in $\mathcal C^\alpha$ when $t_n \to t_0$. Since $E_f$ is compact in $\mathcal C^\alpha$, there exists a subsequence $t_{n_k}\to t_0$ such that $f(t_{n_k}) \to l$ for some $l\in \mathcal C^\alpha$, thus $l\in \mathcal C^{\gamma+}$. On the other hand, $f\in C_T \mathcal C^{\gamma+}$ means that $f(t_n)\to f(t_0)$ in $\mathcal C^{\gamma+}$. Thus by uniqueness of the limit we have $l=f(t_0)$. 
\end{proof}

\begin{remark}\label{rm:inductive}
By similar arguments as  in the proofs of Lemma \ref{lm:compactenss} and Lemma \ref{lm:inductive} we obtain the same characterization for any inductive space  of the form $E= \cup_{N\in \mathbb N} E_N$, where $E_N$ is a Banach space, that is 
\begin{itemize}
\item[(i)]  $K\subset E$ is a compact in $E$ if and only if there exists $N$ such that $K\subset E_N$ and $K$ is a compact in $E_N$;
\item[(ii)]  $C_T E =  \cup_{N\in \N} C_T E_{N}.$
\end{itemize}
\end{remark}

\bibliographystyle{plain}
\bibliography{../../../../BIBLIO_FILE/biblio}
\end{document}